\documentclass[10pt]{amsart}
\usepackage{amsmath}
\usepackage{amssymb, amsfonts, amsmath, amsthm, amscd}
\usepackage{mathrsfs}
\usepackage{color, graphics}
\usepackage[all]{xy}
\usepackage{hyperref}
\usepackage{amsxtra}

\newtheorem{theorem}{Theorem}[section]
\newtheorem{proposition}[theorem]{Proposition}
\newtheorem{lemma}[theorem]{Lemma}
\newtheorem{claim}[theorem]{Claim}

\newtheorem{remark}[theorem]{Remark}

\theoremstyle{remark}

\numberwithin{equation}{section}

\def\deg{\operatorname{deg}}%
\def\dim{\operatorname{dim}}%
\def\max{\operatorname{max}}%
\def\corank{\operatorname{corank}}%
\def\ker{\operatorname{ker}}%

\def\cli{\hbox{\rm Cliff}}
\def\pic{\hbox{\rm Pic}}

\def\deg{\operatorname{deg}}%
\def\dim{\operatorname{dim}}%
\def\max{\operatorname{max}}%
\def\ker{\operatorname{ker}}%

\def\cli{\hbox{\rm Cliff}}

\newcommand \im   {\ensuremath{\mathrm{im}}}

\newcommand \ext {\ensuremath{\mathrm{Ext}}}
\newcommand \Sec {\ensuremath{\mathrm{Sec}}}

\newcommand \lra {\rightarrow}

\def\cli{\mbox{Cliff}}
\def\deg{\mbox{deg}}

\def\cli{\hbox{\rm Cliff}}

\begin{document}

\title[Brill-Noether loci of rank two vector bundles]
	{Brill-Noether loci of rank two vector bundles on a general $\nu$-gonal curve}
\author{Youngook Choi}
\address{Department of Mathematics Education, Yeungnam University, 280 Daehak-Ro, Gyeongsan, Gyeongbuk 38541,
 Republic of Korea }
\email{ychoi824@yu.ac.kr}
\author{Flaminio Flamini}
\address{Universita' degli Studi di Roma Tor Vergata, 
Dipartimento di Matematica, Via della Ricerca Scientifica-00133 Roma, Italy}
\email{flamini@mat.uniroma2.it}
\author{Seonja Kim}
\address{Department of  Electronic Engineering,
Chungwoon University, Sukgol-ro, Nam-gu, Incheon, 22100, Republic of Korea}
\email{sjkim@chungwoon.ac.kr}
\thanks{The first author was supported by Basic Science Research Program through the National Research Foundation of Korea(NRF) funded by the Ministry of Education(NRF-2016R1D1A3B03933342).  The third  author was supported by Basic Science Research Program through the National Research Foundation of Korea(NRF) funded by the Ministry of Education (NRF-2016R1D1A1B03930844)}


\begin{abstract} In this paper we study the Brill Noether locus  of rank 2, (semi)stable vector bundles with at least two sections and of suitable degrees on a general $\nu$-gonal curve. We classify its reduced components whose dimensions are at least the corresponding  Brill-Noether number.  \color{black}  We moreover describe the general member $\mathcal F$ of such components just in terms of extensions of line bundles with suitable {\em minimality properties}, providing information on the birational geometry of such components as well as on the very-ampleness of $\mathcal F$. 
\end{abstract}
\maketitle
\section{Introduction}
Let $C$ denote a smooth, irreducible, complex projective curve of genus $g \geq 2$. As in the statement of \cite[Theorem]{Teixidor1} (cf.  also Theorem \ref{Teixidor} below), $C$ is said to be {\em general} if $C$ is a curve with general moduli (cf.\;e.g.\;\cite{ACGH},\;pp. 214--215). \color{black} Let $U_C(n, d)$ be the moduli space of semistable, degree $d$, rank $n$ vector bundles on $C$ and let $U^s_C(n,d)$ be the open dense  subset of stable bundles (when $d$ is odd, more precisely one has $U_C(n, d)=U^s_C(n,d)$). Let $B^k_{n,d}\subseteq U_C(n, d)$ be the {\em Brill-Noether locus} which consists of vector bundles $\mathcal F$  having $h^0(\mathcal F)\ge k$, for a positive integer $k$.  \color{black}

Traditionally, we denote by $W^k_d$ the Brill-Noether locus  $B^{k+1}_{1,d}$ of line bundles
$L\in \mbox{Pic}^d(C)$ having $h^0(L)\ge k+1$, for a non-negative integer $k$. With little 
abuse of notation, we will sometimes identify line bundles with corresponding divisor classes, 
interchangeably using multiplicative and additive notation. 

For the case of rank $2$ vector bundles, we simply put  $B^k_d:=B^k_{2,d}$, for which it is well-known that the dimension of $B_d^k \cap U^s_C(2,d)$ is at least the Brill-Noether number $\rho_d^{k}:=4g-3-ik$, where $i:=k+2g-2-d$ (cf. \cite{Sun}). This is no longer true for possible components of $B^k_d$ in $U_C(2,d) \setminus U^s_C(2,d)$, i.e. not containing stable points, which can occur only for $d$ even (cf. \cite[Remark\;3.3]{CF} for more explanations and details). \color{black} 

 In the range $0 \le d \le 2g-2$,  $B^1_d$ has been deeply studied on any curve $C$ by several authors (cf. \cite{Sun,L}). \color{black}
Concerning $B^2_d$,  using a degeneration argument, N. Sundaram \cite{Sun} proved  that $B^2_d$ is non-empty for any  $C$ and for odd $d$ such that $g\le d\le 2g-3$. M. Teixidor I Bigas generalizes Sundaram's result as follows: 

\begin{theorem}[\cite{Teixidor1}]\label{Teixidor} Given a non-singular curve $C$ and a $d$, 
 $3\le d\le 2g-1$,  $B^2_d \cap U^s_C(2,d)$ has a component of dimension $\rho^2_d=2d-3$ and a generic point on it corresponds to a vector bundle whose space of sections has dimension $2$ and the generic section has no zeroes. If $C$ is general, this is the only component of $B^2_d\cap U^s_C(2,d)$. Moreover, $B^2_d\cap U^s_C(2,d)$ has extra components 
if and only if $W^1_n$ is non-empty and $\dim  W^1_n\ge d+2n-2g-1$ for some
$n$  with $2n<d$.
\end{theorem} 

Inspired by Theorem \ref{Teixidor}, in this paper we focus on $B^2_d$ for  $C$ a {\em general $\nu$-gonal curve} of genus $g$, i.e. $C$ corresponds to a general point of the $\nu$-gonal stratum $\mathcal  M^1_{g,\nu} \subset \mathcal M_g$. Precisely, we prove the following:
\begin{theorem}\label{thm:main}
Let $C$ be a general $\nu$-gonal ($3\le \nu \le \frac{g+8}{4}$) curve of genus $g$ 
and let $A$ be the unique line bundle of degree $\nu$  and $h^0(A)=2$.
For  any  positive integer $d$ with
$
2+2\nu\le d\le g-3,$ the reduced components of  $B^{2}_d$ having dimension at least $\rho_d^2$ are only two,  which we denote by $B_{\rm reg}$ and $B_{\rm sup}$:
\begin{enumerate}
\item[(i)] $B_{\rm reg}$ is  generically smooth, of dimension $\rho^{2}_d=2d-3$ ({\em regular} for short). Moreover, $\mathcal F$ general in $B_{\rm reg}$ is  stable, fitting in an exact sequence \color{black}
\begin{eqnarray*}
& 0\to\mathcal O_C(p)\to \mathcal F\to  L \to 0,
\end{eqnarray*}
where $p\in C$ and $L \in W^0_{d-1}$ are general and where $h^0(\mathcal F)=2$. 
\item[(ii)] $B_{\rm sup}$ is generically smooth, of dimension $d+2g-2\nu-2 > \rho^2_d$ ({\em superabundant} for short).
Moreover,  $\mathcal F$ general in $B_{\rm sup}$  is stable, fitting in an exact sequence 
$$ 0\to A\to \mathcal F\to L \to 0, $$
where $L$ is a general line bundle of degree $d-\nu$ and $h^0(\mathcal F)=2$. 
\end{enumerate}
\end{theorem}

A more precise statement of this result is given in Theorem \ref{thm:main2}  for  its {\em residual} version (i.e. concerning the isomorphic Brill Noether locus $B^{2g-d}_{4g-4-d}$). \color{black}  Indeed, for any non negative integer $i$, if one sets $k_i := d-2g+2+i$ and $$B_d^{k_i} : = \{\mathcal F \in U_C(2,d)\;|\: h^0(\mathcal F) \ge k_i\} =  \{\mathcal F \in U_C(2,d)\;|\: h^1(\mathcal F) \ge i\},$$one has natural isomorphisms $B_d^{k_i} \simeq B_{4g-4-d}^{i}$, arising from the correspondence $\mathcal F \to \omega_C \otimes \mathcal F^*$, Serre duality and semistability (cf.\; Sect.\;\ref{ss:seg}).  The key ingredients of our approach are the geometric theory of extensions introduced by Atiyah, Newstead, Lange-Narasimhan et al. (cf.\;e.g.\cite{LN}), Theorem \ref{CF} below and suitable parametric computations involving special and effective quotient line bundles and related families of sections of ruled surfaces, which make sense  in the set-up of Theorem \ref{thm:main2}.  Finally, by Theorems  \ref{Teixidor} and  \ref{thm:main}, we can also see that  a general vector bundle in $B_{\rm reg}$ admits  a special section whose zero locus is of degree one while its general section has no zeros (cf. the proof of \cite[Theorem]{Teixidor1} and Remark \ref{rem:min} (ii) below). 

For standard \color{black} terminology, we refer the reader to \cite{H}.

\smallskip

\noindent
{\bf Acknowledgements}. The authors thank KIAS and Dipartimento di Matematica Universita' di Roma "Tor Vergata" for the warm atmosphere and hospitality during the collaboration and the preparation of this article.


\section{Preliminaries}
\subsection{Preliminary results on general $\nu$-gonal curves} In this section we will review \color{black}  some results concerning line bundles on general $\nu$-gonal curves, which will be used in the paper.

\begin{lemma}\label{S2} (cf. \cite[Corollary\;1]{KK}) On a general $\nu$-gonal curve of genus $g \ge 2\nu-2$, with 
$\nu \ge 3$, there does not exist a $g^r_{\nu - 2 + 2r}$ with $\nu - 2 + 2r \le g-1$, $r \ge 2$.  
\end{lemma}

\noindent 
The {\em Clifford index} of  a line bundle $ L$ on a curve $C$ is defined by 
$$\cli(L) := \deg( L) -  2 (h^0( L)-1).$$

\begin{theorem}[{\cite{K}, Theorem 2.1}]\label{S1} Let $C$ be a general $\nu$-gonal curve of genus $g\ge 4$, $\nu\ge 4$, and let $g^1_\nu$ be the unique pencil of degree $\nu$ on $C$. If $C$ has a line bundle $L$ with
$\cli(L)\le \frac{g-4}{2}$ and $\deg L\le g-1$, then 
 $| L|=(\mbox{dim}| L|) g^1_\nu +B$, for some effective divisor $B$. 
\end{theorem}

\subsection{Segre invariant and semistable vector bundles}\label{ss:seg} 
Given a rank $2$ vector bundle $\mathcal F$ on $C$, the \textit{Segre invariant} $s_1(\mathcal F) \in \mathbb{Z}$ of $\mathcal F$ is defined by
\[s_1(\mathcal F) = \min_{N \subset \mathcal F} \left\{   \deg \mathcal F  -   2\; \deg N  \right\},
\]
where $N$ runs through all the sub-line bundles of $\mathcal F$.  It easily follows from the definition that $s_1(\mathcal F) = s_1(\mathcal F \otimes L)$, for any line bundle $L$, and 
$s_1(\mathcal F) = s_1 (\mathcal F^*)$, where $\mathcal F^*$ denotes the dual bundle of $\mathcal F$. A sub-line bundle $ N \subset \mathcal F$ is called a \textit{maximal sub-line bundle} of $\mathcal F$ if $\deg N$ is maximal among all sub-line bundles of $\mathcal F$; in such a case $\mathcal F/N$ is a \textit{minimal  quotient line bundle} of $\mathcal F$, i.e. is of minimal degree among quotient line bundles of $\mathcal F$ . In particular, $\mathcal F$ 
is {\em semistable}  (resp. {\em stable}) if and only if $s_1(\mathcal F) \ge 0$ (resp.   $s_1(\mathcal F) > 0$). 

\subsection{Extensions, secant varieties and semistable vector bundles}
Let $\delta$ be a positive integer. Consider $L\in \pic^\delta(C)$ and $N\in\pic^{d-\delta}(C)$. 
 The extension space $\ext^1(L,N)$ parametrizes isomorphism classes of extensions and any element $u\in\ext^1(L,N)$ gives rise to a degree $d$, rank $2$ vector bundle $\mathcal F_u$, fitting in an exact sequence
\begin{equation}\label{degree}(u):\;\; 0 \to N \to \mathcal F_u \to L \to 0.\end{equation}

\noindent
We fix once and for all the following notation:
\begin{eqnarray}\label{degree1}
j:=h^1(L), & l:=h^0(L)=\delta-g+1+j,\\
r:=h^1(N), & n:=h^0(N)=d-\delta-g+1+r \notag
\end{eqnarray}

In order to get
$\mathcal F_u$ semistable, a necessary condition is 
\begin{equation}\label{eq:neccond}
2\delta-d  \ge    s_1(\mathcal F_u)\ge 0.
\color{black}
\end{equation}
In such a case, the Riemann-Roch theorem gives 
\begin{equation}\label{eq:m}
\dim(\ext^1(L,N))=
\begin{cases}
\ 2\delta-d+g-1\ &\text{ if } L\ncong N \\
\ g\ &\text{ if } L\cong  N.
\end{cases}
\end{equation}
Since we deal with {\em special} vector bundles, i.e. $h^1(\mathcal F_u) >0$, they always admit a special quotient line bundle. Recall the following: 
\begin{theorem}[\cite{CF}, Lemma 4.1]\label{CF}
Let $\mathcal F$ be a semistable, special, rank $2$ vector bundle on $C$ of 
degree $d\ge 2g-2$. 
 Then there exist a special, effective line bundle $L$ on $C$  of degree $\delta \leq d$, $N\in \pic^{d-\delta}(C)$ and $u\in\ext^1(L,N)$ such that  $\mathcal F=\mathcal F_u$  as in \ref{degree}.
\color{black}
\end{theorem}

 Tensor  \eqref{degree}  by $N^{-1}$ and consider $\mathcal G_e:=\mathcal F_u\otimes N^{-1},$ which fits in 
$$ (e):\;\;\; 0\to\mathcal O_C\to\mathcal G_e\to L -N\to 0, $$
where $e\in\ext^1(L - N,\mathcal O_C)$, so $\deg(\mathcal G_e)=2\delta-d$. Then $(u)$ and $(e)$ define the same point in $\mathbb P:=\mathbb P(H^0(K_C+L-N)^*)$.  When the map $\varphi:=\varphi_{|K_C+L-N|}:C\to\mathbb P$ is a morphism, set $X:=\varphi(C)\subset \mathbb P$. 
For any positive integer $h$ denote by $\Sec_h(X)$ the $h^{st}$-secant variety of $X$, 
defined as the closure of the union of all linear subspaces $\langle \varphi(D)\rangle\subset\mathbb P$,
for general divisors $D$ of degree $h$ on $C$. One has
$$\dim(\Sec_h(X))=\min\{\dim(\mathbb P), 2h-1\}.$$ 

\begin{theorem} (\cite[Proposition 1.1]{LN})\label{LN}
Let $2\delta-d\ge 2$;  then $\varphi$ is a morphism and, for any integer
$ s \equiv 2\delta-d\  \text{ (mod\ 2) }$  such that  $ 4+ d-2\delta\le s\le 2\delta-d,$ one has
$$ s_1 (\mathcal E_e)\ge s \Leftrightarrow e\notin \Sec_{\frac{1}{2}(2\delta-d+s-2)}(X). $$
\end{theorem}

\section{The main result}\label{s:mainres}  In this section $C$ will denote a general $\nu$-gonal  curve of genus $g\ge 4$ and $A$ the unique line bundle of degree  $\nu$ with $h^0(A)=2$.  As explained in the Introduction, from now on we will be concerned with the residual version of Theorem \ref{thm:main}; therefore we set 
\begin{eqnarray}\label{eq:ourbounds}
3\le \nu \le \frac{g+8}{4} \;\; {\rm and} \;\;3g-1\le d\le 4g-6-2 \nu, 
\end{eqnarray}where $d$ is an integer. 
For suitable line bundles $L$ and $N$ on $C$, we consider
rank $2$ vector bundles $\mathcal F$ arising as extensions. We will give conditions on $L$ and $N$ under which $\mathcal F$ is general in  a certain component of the Brill-Noether locus $B^{k_2}_d$,  where  $k_2=d-2g+4$ as in Introduction. We moreover show that $L$ is a quotient of $\mathcal F$ with suitable {\em minimality} properties. Finally, we prove the following theorem.

\begin{theorem}\label{thm:main2} 
The reduced components of $B^{k_2}_d$ having dimension at least $\rho_d^{k_2}$ are only two,  which we denote by $B_{\rm reg}$ and $B_{\rm sup}$:
\begin{enumerate}
\item[(i)]
 The component $B_{\rm reg}$ is {\em regular},  i.e. generically smooth and of  dimension $\rho^{k_2}_d=8g-2d-11$. \color{black}   A general element $\mathcal F$ of $B_{\rm reg}$ is stable, fitting in an exact sequence
\begin{equation}\label{exactB0} 
 0\to K_C-D \to \mathcal F\to K_C-p\to 0,  
\end{equation}where $p\in C$ and $D\in C^{(4g-5-d)}$ are general. Specifically,  $s_1(\mathcal F) \ge 1$ (resp., $2$)  if $d$ is odd (resp., even). Moreover, $K_C-p$ is minimal among special quotient line bundles of $\mathcal F$ and $\mathcal F$ is very ample for $\nu \ge 4$;  
\item[(ii)] The component $B_{\rm sup}$ is generically smooth,  of dimension $6g-d-2 \nu -6 > \rho^{k_2}_d$, i.e. $B_{\rm sup}$ is {\em superabundant}. A general element $\mathcal F$ of $B_{\rm sup}$ is stable, very-ample, fitting in an exact sequence
\begin{equation}\label{exactB1}
0\to N\to \mathcal F\to K_C-A\to 0, 
\end{equation}
for $N\in\pic^{d-2g+2+\nu}(C)$ general. Moreover, 
$s_1(\mathcal F)=4g-4 - d - 2\nu$ and $K_C-A$ is a minimal quotient of $\mathcal F$.
\end{enumerate}
\end{theorem} 

\begin{proof} In Sect.\;\ref{ss:superabundant}  and \ref{ss:regular} we will construct the components $B_{\rm sup}$ and $B_{\rm reg}$, respectively, and prove all the statements in Theorem \ref{thm:main2} except for the minimality property of $K_C-p$  in (i) and the uniqueness of  $B_{\rm sup}$ and $B_{\rm reg}$, which will be proved in Sect.\;\ref{ss:noother}.  \end{proof}

\begin{remark}\label{rem:17Jan} {\normalfont (i) As explained in the Introduction, Theorem \ref{thm:main2} and the natural isomorphism $B^{k_2}_d\simeq B^2_{4g-4-d}$ give also a proof of Theorem \ref{thm:main}.

\smallskip

\noindent
(ii)    It is well-known how the study of rank 2 vector bundles on curves is related to that of (surface) scrolls in projective space. Therefore, very-ampleness condition in Theorem \ref{thm:main2}  is a key for the study of components of Hilbert schemes of smooth  scrolls, in a suitable projective space, dominating $\mathcal M^1_{g,\nu}  $. This will be the subject of a forthcoming paper.

 } 
\end{remark}

\subsection{The superabundant component $B_{\rm sup}$}\label{ss:superabundant}  In this section we first construct the component $B_{\rm sup}$ as in Theorem \ref{thm:main2} . We consider the line bundle $L:= K_C-A \in W^{g-\nu}_{2g-2-\nu}$ and a general  $N \in \pic^{d-2g+2+\nu}(C)$;  
since $d-2g+2+\nu \ge g+1+\nu$ from \eqref{eq:ourbounds}, in particular $h^1(N)=0$. We first need the following preliminary result.

\begin{lemma}\label{lem:i=2.2} Let $N \in \pic^{d-2g+2+\nu}(C)$ be general. Then,
for a general $u\in\ext^1(K_C-A, N)$, the corresponding rank $2$ vector bundle $\mathcal F_u$ is  stable with:
\begin{enumerate}
\item[(a)]  $h^1(\mathcal F_u)= h^1(K_C -A) = 2$;
\item[(b)] $s_1(\mathcal F_u)= 4g-4 - 2\nu -d$;  more precisely, $K_C-A$ is a minimal quotient line bundle of  $\mathcal F_u$;
\item[(c)] $\mathcal F_u$ is very ample. 
\end{enumerate} 
\end{lemma}

\begin{proof}  To ease notation, set $L=K_C-A$ and $\delta:=\deg L$. 
To show that $\mathcal F_u$ is stable, note that the upper bound on $d$ in  \eqref{eq:ourbounds} implies  $2\delta-d = 2 (2g-2-\nu) - d \ge 2$; so we are in position to apply  Theorem \ref{LN}. 
 We consider the natural morphism
$$\varphi:=\varphi_{|K_C+L-N|}: C{\longrightarrow} \mathbb P:=\mathbb P(\ext^1(L,N)).$$
Set $X:=\varphi(C)$. Let  $s$  be an integer such that $ s\equiv 2\delta-d\  {\rm{ (mod \;2)}}$ and $ 0<s\le 2\delta - d$.   \color{black}
Since $s\le 2\delta-d=4g-4-2 \nu-d<g-3$, we have 
$$\dim\left(\Sec_{\frac{1}{2}(2\delta-d+s-2)}(X)\right)=2\delta-d+s-3<2 \delta-d+g-2= \dim(\mathbb P),$$where the last equality follows from \eqref{eq:m}  and $L \ncong N$. One can therefore take $s=2\delta-d$, so that the general  $\mathcal F_u$ arising from \eqref{exactB1} is of degree $d$, with $h^1(\mathcal F_u) = h^1(L) = 2$ and it is stable, 
since $s_1(\mathcal F_u) = 2 \delta - d = 4g-4 - 2\nu -d \geq 2$; the equality $s_1(\mathcal F_u) = 2 \delta - d$ follows from Theorem \ref{LN} and from \eqref{exactB1}.  This proves the stability of $\mathcal F_u$ together with (a) and (b).

Finally, to prove (c),  observe first that $K_C-A$ is very ample: indeed, if $K_C-A$ is not very ample, by the Riemann-Roch theorem 
there exists a $g^2_{\nu+2}$ on $C$; this is contrary to Lemma \ref{S2}, since the hypothesis $3 \le \nu \le \frac{g+8}{4}$ implies $g \ge 2\nu-2 + (2\nu - 6) \ge 2 \nu - 2$.  At the same time, since $\deg(N)=d-2g+2+\nu \geq g+4$ by \eqref{eq:ourbounds}, a general $N$ is also very ample. Thus any $\mathcal F_u$ as in \eqref{exactB1} is very ample too. 
\end{proof}

We now want to show that vector bundles constructed in Lemma \ref{lem:i=2.2} fill up the component $B_{\rm sup}$, as $N$ varies in $\pic^{d-2g+2+\nu}(C)$. To do this, we need to consider a parameter space of rank 2 vector bundles on $C$, arising as extensions of $K_C-A$ by $N$, as $N$ varies. If $\mathcal N\to\pic^{d-2g+2+\nu}(C)\times C$ is a Poincar$\acute{e}$ line bundle, we have the following diagram:
\begin{center}
  \begin{picture}(300,100)
    \put(58,13){$\pic^{d-2g+2+\nu}(C)$}
    \put(200,15){$C$}
    \put(115,55){$\pic^{d-2g+2+\nu}(C)\times C$}
    \put(145,50){\vector(-3,-2){40}}
    \put(160,50){\vector(3,-2){40}}
    \put(150,90){$\mathcal N$}
     \put(153,85){\vector(0,-1){15}}
    \put(112,35){$p_1$}
    \put(187,35){$p_2$}
      \put(240,50){$K_C-A$}
        \put(245,45){\vector(-3,-2){30}}
  \end{picture}
\end{center}Set $\mathcal E_{d,\nu}:={R^1p_1}_*(\mathcal N\otimes p_2^*(A-K_C))$. 
By \cite[pp.\;166-167]{ACGH}, $\mathcal E_{d,\nu}$ is a vector bundle on a suitable open, dense subset $S \subseteq \pic^{d-2g+2+\nu}(C)$ of 
rank $\dim{\ext^1(K_C-A,N)}= 5g-5- 2 \nu -d $ as in \eqref{eq:m}, since $K_C-A \ncong N$.  Consider the projective bundle $\mathbb P(\mathcal E_{d,\nu})\to S$, which is the family of 
$\mathbb P\left(\ext^1(K_C-A,N)\right)$'s as  $N$ varies in $S$. One has$$ \dim \mathbb P(\mathcal E_{d,\nu}) 
=g+ (5g-5- 2 \nu -d) -1=6g-6-2 \nu  -d. $$
 Consider the natural (rational) map
 \begin{eqnarray*}
 &\mathbb P(\mathcal E_{d,\nu})\stackrel{\pi_{d,\nu}}{\dashrightarrow} &U_C(2,d) \\
 &(N, u)\to &\mathcal F_u;
 \end{eqnarray*}  from Lemma \ref{lem:i=2.2} we know that  $ \im (\pi_{d,\nu})\subseteq B^{k_2}_d \cap  U^s_C(2,d)$. 
 
\begin{proposition}\label{thm:i=2.3} The closure $B_{\rm sup}$ of \;$\im (\pi_{d,\nu})$ in $U_C(2,d)$ is a generically smooth component of $B^{k_2}_d$,  having dimension $6g-6-2\nu -d$. In particular $B_{\rm sup}$ is {\em superabundant}. 
\end{proposition}
\begin{proof} 
The result will follow once we prove that $$\dim T_{\mathcal F}(B^{k_2}_d)=\dim B_{\rm sup},$$ for a general $\mathcal F$ in $\im(\pi_{d,\nu})$. 
First we claim  that $\dim B_{\rm sup} =6g-6-2 \nu  -d$. Indeed, let $\Gamma\subset F = \mathbb{P}(\mathcal F_u)$ be the
 section corresponding to  the quotient $\mathcal F_u \to\!\!\!\!\to K_C-A$. Its normal bundle is $N_{\Gamma/F}\simeq K_C-A-N$ (cf.\;\cite[Sect.\;V, Prop.\;2.9]{H});  since $N$ is general of degree at least $g+4$ by \eqref{eq:ourbounds},  we have $h^0(K_C-A-N)=0$; in other words $\Gamma$ is an algebraically isolated section of $F$. This guarantees that  $\pi_{d,\nu}$ is generically finite (for more details see the proof of \cite[Lemma\;6.2]{CF} and apply the same arguments). Hence we get $\dim \im(\pi_{d,\nu}) =6g-6-2 \nu  -d$. 

Now we prove that $\dim T_{\mathcal F}(B^{k_2}_d) =6g-6-2 \nu  -d$.
To show this, consider the Petri map of a general $\mathcal F\in\im(\pi_{d,\nu})$:
$$ \mu_\mathcal F: H^0(\mathcal F)\otimes H^0(\omega_C\otimes \mathcal F^*)\to H^0(\omega_C\otimes \mathcal F\otimes \mathcal F^*). $$By \eqref{exactB1} and  $h^1(N)=0$, we have
$$H^0(\mathcal F)\simeq H^0(N)\oplus H^0(K_C-A) \;\;\; {\rm and} \;\;\; 
H^0(\omega_C\otimes \mathcal F^*)\simeq H^0(A).$$Thus $\mu_\mathcal F$ reads as 
\begin{equation*}
\begin{CD}
\left(H^0(N)\oplus H^0(K_C-A)\right)&\;\otimes\;& \;H^0(A)&\;\; \stackrel
{\mu_\mathcal F }{\longrightarrow}\;\; &H^0(\omega_C\otimes \mathcal F\otimes \mathcal F^*).\\
\end{CD}
 \end{equation*}  Consider the following natural multiplication maps:

 \begin{eqnarray}
 \mu_{A,N}:& H^0(N) \otimes  H^0(A)\to H^0(N+A)\label{muA}\\
 \mu_{0,A}: &H^0(K_C-A) \otimes H^0(A) \to H^0(K_C)\label{mu0}. 
 \end{eqnarray}
\color{black}
\begin{claim}\label{cl:ker}  $\ker(\mu_\mathcal F)\simeq \ker(\mu_{0,A})\oplus \ker(\mu_{A,N}) $. 
\end{claim}
\begin{proof}[Proof of Claim \ref{cl:ker}] Consider the exact diagram:
  
\begin{equation*}\label{eq1b}
\begin{array}{ccccccccccccccccccccccc}
&&0&&0&&0&&\\[1ex]
&&\downarrow &&\downarrow&&\downarrow&&\\[1ex]
0&\lra& N + A - K_C & \rightarrow & \mathcal F\otimes (A-K_C)& \rightarrow & \mathcal O_C &\rightarrow & 0 \\[1ex]
&&\downarrow && \downarrow && \downarrow& \\[1ex]
0&\lra & N \otimes  \mathcal F^*& \rightarrow & \mathcal F\otimes \mathcal F^* & \rightarrow &
\mathcal  (K_C-A) \otimes F^*&\rightarrow & 0 \\[1ex]
&&\downarrow &&\downarrow &&\downarrow  &&\\[1ex]
0&\lra & \mathcal O_{C}& \lra & \mathcal F\otimes N^{-1}&\lra& (K_C-A)\otimes N^{-1}&0& \\[1ex]
&&\downarrow &&\downarrow&&\downarrow&&\\[1ex]
&&0&&0&&0&&;
\end{array}
\end{equation*}
 \color{black}  which arises from  \eqref{exactB1} and its dual sequence $0\to A-K_C\to \mathcal F^*\simeq\mathcal F(A-K_C-N)\to N^{-1}\to 0$. If we tensor the column in the middle by  $\omega_C$, we get \color{black} $H^0(\mathcal F\otimes A)\hookrightarrow  H^0(\omega_C\otimes \mathcal F\otimes \mathcal F^*)$. 

Observe moreover that $ H^0(N+A)\oplus H^0(K_C)\simeq H^0(\mathcal F\otimes A)$, which follows from \eqref{exactB1} tensored by $A$ and the fact that $h^1(N+A)=0$. Therefore there is no intersection between $\im(\mu_{0,A})$ and $\im(\mu_{A,N})$ and the statement is proved. 
\end{proof}

\noindent 
By Claim \ref{cl:ker}, 
\begin{eqnarray*}
\dim T_{\mathcal F}(B^{k_2}_d)&=&4g-3-h^0(\mathcal F)h^1(\mathcal F)+\dim(\ker \mu_\mathcal F) \\
&=&4g-3-2(d-2g+4)+\dim(\ker(\mu_0(A)))+\dim(\ker(\mu_{A,N})).
\end{eqnarray*} From \eqref{muA} and \eqref{mu0}, we have
$$\ker(\mu_{0,A})\simeq H^0(K_C-2A)\cong H^1(2A)^*\;\; {\rm and} \;\;   \ker(\mu_{A,N})\simeq H^0(N-A),$$
as it follows from the base point free pencil trick. Under the numerical assumption $\nu \le \frac{g+8}{4}$,  from Theorem \ref{S1} we have
$h^0(2A)=3$,  which implies $h^1(2A)=g + 2 -2 \nu$. The inequality $\deg N \ge g+ 1 +\nu$ given by \eqref{eq:ourbounds} and the generality of $N$ show that $h^1(N-A)=0$, which yields $h^0(N-A)=d-3g+3$. So we have
\begin{eqnarray*}
\dim T_{\mathcal F}(B^{k_2}_d)&=&4g-3-2(d-2g+4)+(g + 2 -2 \nu) + (d-3g+3)\\
&=&6g-6-2 \nu -d = \dim B_{\rm sup}.
\end{eqnarray*} To complete the proof, it suffices to observe that  
$\rho_d^{k_2} = 8g-11-2d \le 5g-10-d < 6g-6-2\nu -d,
$ as it follows by \eqref{eq:ourbounds}. \end{proof}


 \subsection{The regular component $B_{\rm reg}$} \label{ss:regular}  In this subsection we construct the regular component $B_{\rm reg}$ as in Theorem \ref{thm:main2}. In what follows, we use \color{black} notation as in \eqref{degree1}, i.e. 
 $l = h^0(L),\;j =h^1(L),\;r= h^1(N)$ which will be considered all positive (cf. Theorem \ref{CF} for $L$).  \color{black} For any exact sequence $(u)$ as in \eqref{degree}, let 
 $\partial_u : H^0(L) \to H^1(N)$ be the corresponding coboundary map.  For any integer $t>0$, consider
\begin{equation}\label{W1}
 \mathcal W_t:=\{u\in\ext^1(L,N)\ |\ {\rm corank} (\partial_u)\ge t\}\subseteq \ext^1(L,N), 
\end{equation}
 which has a natural structure of determinantal scheme; its expected codimension is
  $t(l-r+t)$ \color{black} (cf. \cite[Sect.\,5.2]{CF}). In this set--up, one has: \color{black}

\begin{theorem}(\cite[Theorem 5.8 and Corollary 5.9]{CF})\label{CF5.8} Let $C$ be a smooth curve of genus $g\ge 3$.
 Let
$$ r= h^1 (N) \ge1,\ l =h^0 (L) \ge\max\{1,r-1\},\ m:=\dim(\ext^1(L,N))\ge l+1.$$
\color{black}
Then, we have:
\begin{enumerate}
\item[(i)]  $l-r+1\ge 0$; \color{black}
\item[(ii)] $\mathcal W_1$ is irreducible of (expected) dimension $m- (l-r+1)$\color{black}; 
\item[(iii)] if $l \geq r$, then $\mathcal W_1 \subset \ext^1(L,N)$. Moreover for general $u \in  \ext^1(L,N)$,  $\partial_u$ is surjective whereas for general $w \in  \mathcal W_1$,  ${\rm corank} (\partial_w)=1$. 
\end{enumerate}
\end{theorem}

To construct $B_{\rm reg}$, observe firts that by \eqref{eq:ourbounds}  $W^0_{4g-5-d}$ is not empty, irreducible and $h^0(D) =1$, for  general 
$D \in W^0_{4g-5-d}$. We will prove the following preliminary result. 

\begin{lemma}\label{lem:i=1.2} Let $D \in W^0_{4g-5-d}$ and  $p\in C$ be general and let $\mathcal W_1 \subseteq \ext^1(K_C-p,K_C-D)$ be as in \eqref{W1}.  Then, for $u\in \mathcal W_1 $  general, \color{black}
the corresponding rank $2$ vector bundle $\mathcal F_u$ is stable, \color{black} with:
\begin{enumerate}
\item[(a)] $h^1(\mathcal F_u)=2$;
\item[(b)] $s_1(\mathcal F) \ge 1$ (resp., $2$)  if $d$ is odd (resp., even);  
\item[(c)] $\mathcal F_u$ is very ample when $\nu \geq 4$. 
 \end{enumerate} 
\end{lemma}

\begin{proof} From the assumptions we have: 
\begin{equation}\label{degree0}
\begin{CD}
&(u)& : 0\to &K_C-D&\to &\ \ \mathcal F\ \ & \to  \ &\ \ K_C-p\ \ &\to 0\\
&\deg&       &d-2g+3&& d&& 2g-3&\\
&h^0&       &d-3g+5&&  && g-1&\\
&h^1&       &1&&  && 1&
\end{CD}
\end{equation} By \eqref{eq:ourbounds}  $\deg D = 4g-d-5 \ge 2\nu + 1$, therefore 
$K_C-D \ncong K_C-p$; thus, using  \eqref{eq:m}  and notation as in Theorem \ref{CF5.8}, one has$$l=g-1, \; r=1 \;\; {\rm and} \;\; m=\dim \ext^1(K_C-p,K_C-D)=5g-7-d.$$By \eqref{eq:ourbounds}, one has $d \le 4g-7$ so  $m\ge l+1=g$. Hence we can apply Theorem \ref{CF5.8} to 
$$\mathcal W_1=\{u\in\ext^1(K_C-p, K_C-D)\ |\ \corank(\partial_u)\ge 1\},$$which therefore is irreducible, of  (expected) dimension $ \dim \mathcal W_1 = m- 1(l-r+1)=4g-6-d$.  Moreover, by Theorem \ref{CF5.8} (iii) and formula \eqref{degree0}, for general $u \in \mathcal W_1$ one has $h^1(\mathcal F_u) =2$, which proves (a).  

We now want to show that  $\mathcal F_u$ satisfies also (b), for $u \in \mathcal W_1$ general; in particular it is stable. To do this, set $\mathbb{P}:= \mathbb{P}\left(\ext^1(K_C-p, K_C-D)\right)$ and 
consider the projective scheme $\widehat{\mathcal W}_1 := \mathbb{P} (\mathcal W_1) \subset\mathbb P$, which has therefore dimension $4g-7-d$. Posing $\delta := 2g-3$ and considering \eqref{eq:ourbounds}, 
one has $2\delta - d \geq 2 \nu \geq 6$. We are therefore in position to apply  Theorem \ref{LN}. 
 We consider the natural morphism
$ C\stackrel{\varphi}{\longrightarrow}\mathbb P$, given by the complete linear system $|K_C+D-p|$. Set $X = \varphi\left(C\right)$, as in the proof of Lemma \ref{lem:i=2.2}. Let $s$  be an integer such that  $ s\equiv 2\delta-d\  {\rm{ (mod \;2) \ \ and  }} $ $ 0 \le s\le 2\delta - d$. \color{black}
Then we have $$\dim \Sec_{\frac{1}{2}(2\delta-d+s-2)}(X) =2\delta - d + s - 3 = 4g - 9 - d + s \le 4g-7-d = \dim \widehat{\mathcal W}_1$$if and only if 
$s\le 2$, where the equality holds if and only if $s=2$. 

Therefore, for $d$ odd,  by Theorem \ref{LN} one has $s_1(\mathcal F_u) \geq 1$ for $u \in \mathcal W_1$ general; in particular $\mathcal F_u$ is stable and (b) is proved in this case. 

For $d$ even, if one dualizes the exact sequence \eqref{exactB0} and tensors via $\omega_C$, one gets 
$$(e): \;\; 0 \to p \to \mathcal E_e:= \mathcal F_u^* \otimes \omega_C  \to D \to 0,$$where $(e)$ defines the same point as \color{black}  $(u)$ in the projective space $\mathbb P$; in particular $s_1 (\mathcal F_u) = s_1 (\mathcal E_e)$ (cf. Sect.\,\ref{ss:seg}) and $h^0(\mathcal E_e) = 2$, by Serre duality and the fact that $(u) \in \widehat{\mathcal W}_1$. Following the same strategy as in the first part of the proof of \cite[Theorem]{Teixidor1}, one deduces that $(e)$ belongs to the linear span  $\langle \varphi(D) \rangle \subset \mathbb P$. On the other hand, any point $x \in \langle \varphi(D) \rangle $ gives rise to an extension:
$$(x):\;\; 0 \to p \to \mathcal E_x \to D \to 0$$which belongs to $\widehat{\mathcal W}_1$, since $h^0(\mathcal E_x) = 2$ (cf.\;diagram (2) and the subsequent details in the proof of \cite[Theorem]{Teixidor1}). Thus $\langle \varphi(D) \rangle \subseteq \widehat{\mathcal W}_1$. By the Riemann-Roch theorem, 
$$\dim \langle \varphi(D) \rangle = h^0(K_C+D-p) - h^0(K_C-p) -1 = 4g-7-d = \dim \widehat{\mathcal W}_1.$$Since they are both closed and irreducible, one gets $\widehat{\mathcal W}_1= \langle \varphi(D) \rangle$. On the other hand  
$$\Sec_{\frac{1}{2}(2\delta-d+2-2)}(X) = \Sec_{\frac{1}{2}(4g-6-d)}(X),$$which is of dimension $4g-7-d$ too, is non-degenerate in $\mathbb P$  as $X \subset \mathbb P$ is not. Thus, we conclude that 
$\widehat{\mathcal W}_1 \neq  \Sec_{\frac{1}{2}(4g-6-d)}(X)$. In particular, from Theorem \ref{LN},  for a general $u \in \widehat{\mathcal W}_1$ one has $s_1(\mathcal F_u ) \geq 2$, so $\mathcal F_u$ is stable and (b) is proved also in this case.

To prove (c) observe first that, since $\nu \geq 4$ by assumption, then $K_C-p$ is very ample as it follows by the Riemann-Roch theorem. Now: 

\begin{claim}\label{cl:vaK-D}  For general $D \in W^0_{4g-5-d}$, $K_C-D$ is very ample if $\nu \ge 4$. 
\end{claim}
\begin{proof} [Proof of Claim \ref{cl:vaK-D}] Assume by contradiction that $K_C-D$ is not very ample for  general $D\in W^0_{4g-5-d}$.  For a non-negative integer $\tau$, define the following:
$$\Xi _\tau  := \{ (D, p+q) \in W^0_{4g-5-d} \times W^0_2 ~|~  h^0 (D+p+q) =\tau+1 \}. $$If $\Xi_\tau \neq \emptyset$,
\color{black} then we have the diagram:
\begin{center} 
   \begin{picture}(300,60)
    \put(95,10){$W^0_{4g-5-d}$}
    \put(215,10){$W^\tau_{4g-3-d}$}
    \put(170,55){$\Xi _\tau$}
    \put(165,50){\vector(-3,-2){40}}
    \put(180,50){\vector(3,-2){40}}
        \put(127,40){$\pi_\tau$}
    \put(210,40){${\wp}_\tau$}
  \end{picture}
\end{center}
which is given by $\pi _\tau (D, p+q) :=D$ and $\wp _\tau (D, p+q) := D+p+q$.
The assumption implies that, for some $\tau \in\{1,2\}$, the image of $ \pi _\tau$ is dense in  $W^0_{4g-5-d}$. Considering the map $\wp _\tau$, we get $\mbox{dim} \Xi _\tau \leq \mbox{dim}W^\tau_{4g-3-d}+\tau$. By Martens' and Mumford's Theorems (cf. \cite[Thm. (5.1), (5.2)]{ACGH}), we have $\mbox{dim}W^\tau_{4g-3-d}\leq 4g-5-d-2\tau$, since $C$ is  a general $\nu$-gonal curve with $ \nu\geq 4$ and $4g-3-d\leq g-2$ by \eqref{eq:ourbounds}. In sum,  it turns out that  
$$\mbox{dim}W^0_{4g-5-d} \leq \mbox{dim} \Xi _r\leq 4g-5-d-\tau, $$
which cannot occur. This completes the  proof of the claim.\end{proof} 

\noindent
The above arguments prove (c) and complete the proof of the Lemma.  \end{proof}

To construct the component $B_{\rm reg}$ notice that, as in Sect.\;\ref{ss:superabundant}, one has a projective bundle $\mathbb P(\mathcal E_{d})\to S$ where $S \subseteq W^0_{4g-5-d}\times C$ is a suitable open dense subset: $\mathbb P(\mathcal E_{d})$ is the family of 
$\mathbb P(\ext^1(K_C-p,K_C-D))$'s  as  $(D, p) \in S$ varies. Since, for any such $(D, p) \in S$, 
$\widehat{\mathcal W}_1$ is irreducible of constant dimension $4g-7-d$, one has an irreducible subscheme 
$\widehat{\mathcal W}_1^{Tot} \subset \mathbb P(\mathcal E_{d})$ which has therefore dimension 
$$\dim \widehat{\mathcal W}_1^{Tot} = \dim S + 4g - 7 - d = 4g - d - 4 + 4g - 7 - d = 8 g - 2d - 11 = \rho_d^{k_2}.$$
From Lemma \ref{lem:i=1.2}, one has the natural (rational) map
 \[
 \begin{array}{ccc}
 \widehat{\mathcal W}_1^{Tot}& \stackrel{\pi}{\dashrightarrow} & U_C(d) \\
 (D,p, u) & \longrightarrow &\mathcal F_u;
 \end{array}
 \] and  $\im(\pi) \subset B^{k_2}_d \cap U_C^s(2,d)$.

\begin{proposition}\label{thm:i=1.3} The closure $B_{\rm reg}$ of $\im(\pi)$ in $U_C(2,d)$ is a generically smooth component of $B^{k_2}_d$ with dimension $\rho_d^{k_2} = 8g-11-2d$, i.e. $B_{\rm reg}$ is {\em regular}.  
\end{proposition}

\begin{proof} From  the fact that $\im(\pi)$ contains stable bundles, any component of $B^{k_2}_d$ containing it has dimension at least $\rho_d^{k_2}$. We concentrate in computing $\dim T_{\mathcal F}(B^{k_2}_d)$, for  general $\mathcal F \in \im(\pi)$. Consider the Petri map 
$$\mu_\mathcal F: H^0(\mathcal F)\otimes H^0(\omega_C\otimes \mathcal F^*)\to H^0(\omega_C\otimes \mathcal F\otimes \mathcal F^*)$$for a general $\mathcal F\in \im(\pi)$. From diagram \eqref{degree0} and the fact that $\mathcal F = 
\mathcal F_u$, for some $u$ in some fiber $\widehat{\mathcal W}_1$ of $\widehat{\mathcal W}_1^{Tot}$, one has 
that the corresponding coboundary map $\partial_u$ is the zero-map; in other words 
$$H^0(\mathcal F) \cong H^0(K_C-D) \oplus H^0(K_C-p) \;\;\;{\rm and} \;\;\; H^1(\mathcal F) \cong  H^1(K_C-D) \oplus H^1(K_C-p).$$This means that, for any such bundle, the domain of the Petri map $\mu_{\mathcal F}$ coincides with that 
of $\mu_{{\mathcal F}_0}$, where  
$\mathcal F_0 := (K_C-D) \oplus (K_C-p)$ corresponds to the zero vector in $\mathcal W_1 \subset \ext^1(K_C-p, K_C-D)$. We will concentrate on $\mu_{{\mathcal F}_0}$; observe that 
\[\begin{array}{ccl}
H^0(\mathcal F_0) \otimes H^0(\omega_C \otimes \mathcal F_0^*) & \cong &  \left(H^0(K_C-D) \otimes H^0(D)\right)  \oplus \left(H^0(K_C-D) \otimes H^0(p)\right) \oplus\\
& & \left(H^0(K_C-p) \otimes H^0(D)\right) \oplus \left(H^0(K_C-p) \otimes H^0(p)\right). 
\end{array}
\]Moreover 
$$\omega_C \otimes \mathcal F_0 \otimes \mathcal F_0^* \cong K_C \oplus (K_C+p-D) \oplus (K_C+D-p) \oplus K_C.$$Therefore, for Chern classes reason,  
$$\mu_{\mathcal F_0} = \mu_{0, D} \oplus \mu_{K_C-D, p} \oplus \mu_{K_C-p, D} \oplus \mu_{0, p}$$where the maps 
\begin{eqnarray*}
\mu_{0,D}: & H^0(D)\otimes H^0(K_C-D)\to H^0(K_C),\\
 \mu_{K_C-D,p}:& H^0(K_C-D)\otimes H^0(p)\to H^0(K_C-D+p)\\
 \mu_{K_C-p,D}: & H^0(K_C-p)\otimes H^0(D)\to H^0(K_C+D-p)\\
  \mu_{0,p}: & H^0(p)\otimes H^0(K_C-p)\to H^0(K_C)
 \end{eqnarray*}are natural multiplication maps. Since $h^0(D) = h^0(p) =1$, the maps $\mu_{0,D}, \; \mu_{K_C-D,p},\; \mu_{K_C-p,D}\;\mu_{0,p}$ are all injective and so is $\mu_{\mathcal F_0}$. \color{black}  By semicontinuity on $\mathcal W_1$, one has that $\mu_{\mathcal F}$ is injective, for $\mathcal F$ general in $\widehat{\mathcal W}_1$. 
 
 The previous argument shows that a general $\mathcal F \in \im(\pi)$  is contained in 
only one irreducible component, say $B_{\rm reg}$, of $B_d^{k_2}$ for which 
\begin{eqnarray*}
\dim B_{\rm reg}= \dim T_{\mathcal F}(B_{\rm reg})&=&4g-3-h^0(\mathcal F)h^1(\mathcal F) \\
&=&4g-3-2(d-2g+4) = 8g-11-2d, 
\end{eqnarray*} i.e.  $B_{\rm reg}$ is generically smooth and of dimension $\rho_d^{k_2}$. \color{black}

To conclude that $B_{\rm reg}$ is the closure of $\im(\pi)$, it suffices to show that the rational map $\pi$ is generically finite onto its image. To do this, let $F = \mathbb{P} (\mathcal F_u)$ be the ruled surface, for general $\mathcal F_u 
\in \widehat{\mathcal W}_1^{Tot} $, and let $\Gamma$ be the section corresponding to the 
quotient $\mathcal F_u \to \!\!\!\! \to K_C-p$. Then its normal bundle is $N_{\Gamma/F} \simeq D-p$ which has no sections. Thus, one deduces the generically finiteness of $\pi$ by reasoning as
in the proof of Proposition \ref{thm:i=2.3}.  \end{proof}

\subsection{No other reduced components of dimension at least $\rho_d^{k_2}$}\label{ss:noother}  In this section, we will show that no other  reduced components of $B_d^{k_2} $, having dimension at least $\rho_d^{k_2} = 8g - 11 - 2d$, \color{black} exist except for $B_{\rm reg}$ and $B_{\rm sup}$ constructed in the previous sections. 

Let $B \subset B^{k_2}_d$ be any reduced component with $\dim B \ge \rho_d^{k_2} = 8g-11-2d$; 
from Theorem \ref{CF}, $\mathcal F\in B$ general fits in an exact sequence of the form
\begin{equation}\label{exact_seq}
 0\to N\to \mathcal F\to L \to 0, 
\end{equation}where $L$ is a special, effective line bundle of degree $\delta \leq d$, i.e. $l,\;j>0$ and $h^1(\mathcal F) \ge 2$.

We first focus on the case of $h^1(\mathcal F) =2$. We start with the following:

\begin{proposition}\label{lem:i=2.1} Let $B$ be any reduced  component of $B^{k_2}_d$,  with $\dim B \geq \rho_d^{k_2}$. For  
$\mathcal F$ general in $B$, assume that it fits in an exact sequence like \eqref{exact_seq}, 
with $h^1(\mathcal F)=h^1(L)=2$. Then, $B$ coincides with the component $B_{\rm sup}$ as in Sect.\;\ref{ss:superabundant}.
\end{proposition}

\begin{proof}  Since $\mathcal F$ is  semistable, from \eqref{eq:neccond} and \eqref{eq:ourbounds} one has  
$\deg L\;\ge \frac{3g-1}{2}$. Moreover, since $C$ is a general $\nu$-gonal curve and $h^1(L)=2$, from \cite[Theorem\;2.6]{AC} we have
$|\omega_C\otimes L^{-1}|=g^1_{\nu}+B_b$, where $B_b$ is a base locus of degree $b$.
Hence $L\simeq K_C -A-B_b,$ where $b\le\frac{g-3}{2}-\nu$.
For simplicity, put $\delta := \deg L = 2g-2 - \nu - b$ so $\deg N = d-\delta$.

Since $B$ is reduced, one must have  
$$\dim B = \dim T_{\mathcal F} B$$ 
for  general $\mathcal F \in B$.  We will prove the Proposition  by showing that $\dim B = \dim T_{\mathcal F} B$  can occur only if $L=K_C-A$ and $N$ is non-special, general of its degree.

\begin{claim} \label{Claim 1.}  $\dim B \leq  \begin{cases}
 6g-d-2\nu -6-b  &\text{ if } h^1 (N)=0\\
9g -2d -3\nu -2r -2b-7\ &\text{ if } h^1 (N) \geq 1
\end{cases}$
\end{claim}

\begin{proof} [Proof of Claim \ref{Claim 1.}]   We will use  notation as 
in \eqref{degree1}.  Since  $B$ is irreducible, all integers in \eqref{degree1}  are constant for a general $\mathcal F \in B$. From \eqref{exact_seq} combined  with $L=K_C -A-B_b$, it follows there exists an open dense subset $S$  of a closed subvariety of ${\rm Pic}^{d -\delta} \times C^{(b)}$ and a projective bundle $\mathcal P \to S$, whose general fiber 
identifies with $\mathbb P=\mathbb P(H^0(K_C+L-N)^*) = \mathbb{P} ({\rm Ext}^1(L,N)) \cong \mathbb{P}^{m-1}$, where $m:=\dim ({\rm Ext}^1(L,N))$. Since $h^1(\mathcal F) = h^1(L)$, as in \cite[Sect.\;6]{CF}, the component $B$ has to be the image of $ \mathcal P$  via a dominant rational map 
\[\begin{array}{ccc}
\mathcal P & \stackrel{\pi}{\dasharrow} & B \subset B_{d}^{k_2}\\
\downarrow & & \\
S & & 
\end{array}
\](cf.\;\cite[Sect.\;6]{CF} for details).  Therefore  we obtain  $\dim B \leq  \dim \mathcal P= \dim S + m-1$  since $\mathcal P$ is a projective bundle over $S$ whose general fiber is  $(m-1)$-dimensional.
Specifically,  if  $r\geq 1$ then $S$ is a subset of  $W^{d- \delta -g +r}_{d-\delta}\times C^{(b)}$, \color{black} the latter being equivalent  to $ W^{r-1}_{2g-2+\delta -d}\times C^{(b)}$ by Serre duality, and   $\dim W^{r-1}_{2g-2+\delta -d} \leq 2g-2+\delta-d-2(r -1)$ by using   Martens' theorem (cf.\;\cite[Theorem\;5.1]{ACGH}) for $r\geq 2$. 
Therefore,  we get
$$\dim S \leq
\begin{cases}
g+b &\text{ if } r=0\\
 2g-2+\delta-d-2r+2+b \ &\text{ if } r \geq 1.
\end{cases}$$ 
This inequality, combined  with   \eqref{eq:m}, gives
$$\dim B \leq  \begin{cases}
(g+b )+2\delta -d +g-2   &\text{ if } r=0\\
( 2g-2+\delta-d-2r+2+b) +2\delta -d +g -1\ &\text{ if } r \geq 1,
\end{cases}$$
 since a non-special line bundle cannot be isomorphic to a special one. By substituting  $\delta =2g-2 -\nu -b$, we get the conclusion of Claim \ref{Claim 1.}.
\end{proof}

\begin{claim} \label{Claim 2.}  $\dim T_\mathcal F(B)\ge 6g-d-2\nu-2r-6$ 
\end{claim}
\begin{proof} [Proof of Claim \ref{Claim 2.}]
 The tangent space  $T_\mathcal F(B)$ is the orthogonal space to the image
of  the Petri map: 
$$ \mu_\mathcal F : H^0(\mathcal F)\otimes H^0(\omega_C\otimes \mathcal F^*)
\to H^0(\omega_C\otimes \mathcal F^*\otimes \mathcal F),$$so
$\dim T_\mathcal F(B)=\dim (\im(\mu_\mathcal F)^\perp)= 
h^0(K_C\otimes F^*\otimes F)-h^0(\mathcal F)h^1(\mathcal F)+\dim \ker \mu_\mathcal F$.

From the exact sequence \eqref{exact_seq}, we get $H^0(\mathcal F)\simeq H^0(N)\oplus W$ 
where $W:= \im (H^0(\mathcal F)\to H^0(L))$. 
Since $H^1(\mathcal F)\simeq H^1(L)$, the connecting
homomorphism in \eqref{exact_seq} is surjective, hence $\dim W=l-r= h^0(L)- h^1 (N)$. \color{black}
Let $\mu_{N,\omega_C\otimes L^{-1}}$ and $\mu_{0,W}$ be the maps defined as follows:
\begin{eqnarray*}
\mu_{N,\omega_C\otimes L^{-1}}  & : & H^0(N)\otimes H^0(\omega_C\otimes L^{-1})\to H^0(N\otimes \omega_C\otimes L^{-1})\\
\mu_{0,W} & : & W \otimes H^0(\omega_C \otimes L^{-1})\hookrightarrow H^0(L)\otimes H^0(\omega_C\otimes L^{-1})\to H^0(\omega_C) 
\end{eqnarray*}
Then we have 
\begin{equation}\label{kernel}
 \dim\ker\mu_\mathcal F\ge \dim \ker \mu_{N,\omega_C\otimes L^{-1}}+\dim\ker\mu_{0,W}
 \end{equation}
by the following commutative diagram:
{\small
\begin{center}
  \begin{picture}(380,100)
    \put(10,95){$H^0(\mathcal F)\otimes H^0(\omega_C\otimes \mathcal F^*)$}
    \put(130,98){\vector(1,0){60}}
    \put(157,101){$\mu_{\mathcal F}$}
    \put(260,95){$H^0(\omega_C\otimes \mathcal F\otimes\mathcal F^*)$}
    \put(45,55){$\wr\hspace{-.1cm}\parallel$}
    \put(-10,10){($H^0(N)\oplus W)\otimes H^0(\omega_C\otimes L^{-1})$}
     \put(200,10){($H^0(N)\otimes H^0(\omega_C\otimes L^{-1}))\oplus (W\otimes H^0(\omega_C\otimes L^{-1}))$,}
     \put(130,12){\vector(1,0){60}}
     \put(320,52){$H^0(\omega_C)$}
     \put(210,52){$H^0(\omega_C\otimes L^{-1}\otimes N)$}
     \put(270,32){$\mu_{N,\omega_C\otimes L^{-1}}$}
     \put(340,32){$\mu_{0,W}$}
     \put(265,22){$\vector(0,1){20}$}
     \put(335,22){$\vector(0,1){20}$}
     \put(265,65){$\vector(0,1){20}$}
     \put(335,65){$\vector(0,1){20}$}
     \put(270,70){$\alpha$}
     \put(340,70){$\beta$}
     \put(157,15){$\cong$}
  \end{picture}
\end{center}
}
\noindent 
where the map $\beta$ comes from the trivial section of $H^0(\mathcal F \otimes \mathcal F^*)$ after tensoring via $\omega_C$; to explain the map $\alpha$, if one takes the diagram 
determined by the exact sequence \eqref{exact_seq} and its dual sequence and tensor it by $\omega_C$, one gets: 
\begin{equation*}
\begin{array}{ccccccccccccccccccccccc}
&&0&&0&&0&&\\[1ex]
&&\downarrow &&\downarrow&&\downarrow&&\\[1ex]
0&\lra& \omega_C \otimes N \otimes L^{-1}& \rightarrow & \omega_C \otimes \mathcal F\otimes L^{-1}& \rightarrow & \omega_{C}&\rightarrow & 0 \\[1ex]
&&\downarrow && \downarrow && \downarrow& \\[1ex]
0&\lra &  \mathcal \omega_C \otimes  N \otimes\mathcal F^*& \rightarrow & \omega_C \otimes \mathcal F\otimes \mathcal F^* & \rightarrow & \omega_C \otimes L \otimes \mathcal F^* &\rightarrow & 0\\[1ex]
&&\downarrow &&\downarrow &&\downarrow  &&\\[1ex]
0&\lra & \omega_{C}& \lra &\omega_C\otimes   \mathcal F \otimes N^{-1}&\lra& \omega_C \otimes L \otimes N^{-1}&\rightarrow & 0 \\[1ex]
&&\downarrow &&\downarrow&&\downarrow&&\\[1ex]
&&0&&0&&0&&;
\end{array}
\end{equation*}
\color{black}
the map $\alpha$ is the composition of the two injections 
$$H^0(\omega_C \otimes N \otimes L^{-1}) \hookrightarrow H^0(\omega_C \otimes \mathcal F\otimes L^{-1}) \hookrightarrow H^0(\omega_C \otimes \mathcal F\otimes \mathcal F^*).$$

Since $K_C - L = A + B_b$, by the base point free pencil trick, we have
\begin{eqnarray*}
\dim \ker\mu_{N,\omega_C\otimes L^{-1}}  =  h^0(N-A) & = &  \deg N-\deg A -g +h^0(K_C-N+A)+1 \\\notag
&\ge& d-\delta - \nu  - g + 1  =  d - 3g + 3 + b.\notag
\end{eqnarray*}
From  $\dim W=h^0(L)-r$, it follows that $\dim\ker\mu_{0,W}\ge\dim\ker\mu_0(L) - 2r $, where
$$\mu_0(L) : H^0(L) \otimes H^0(K_C-L) \to H^0(K_C).$$
To compute $\dim \ker\mu_0 (L)$, we  apply once  again the base point free pencil trick which gives 
\begin{eqnarray*}
\dim \ker\mu_0 (L) & = &   h^0(L-A) = h^0(K_C - 2A -  B_b) \\
 & = & 2g-2 - 2\nu - b - g + h^0(2A+B_b) +1 \\
   & \geq & g-2\nu  - b + 2,
\end{eqnarray*}the latter inequality following from the fact that $h^0(2A+B_b) \geq 3$. Hence, from \eqref{kernel}, one has:
\begin{eqnarray*}
\dim \ker\mu_{\mathcal F} & \ge &   d - 3g + 3 + b + g-2\nu  - b + 2 - 2 r  \\
&=& d-2g-2\nu-2r+5. 
\end{eqnarray*}The previous inequality gives  $\dim T_\mathcal F(B)\ge 6g-d-2\nu-2r-6$, proving Claim \ref{Claim 2.}.
\end{proof}
Assume that  $h^1 (N) \geq 1$. Then,
 Claims \ref{Claim 1.}, \ref{Claim 2.}  and  \eqref{eq:ourbounds} imply that 
$$\dim T_{\mathcal F} B - \dim B \geq d-3g +\nu +2b +1 \geq \nu +2b.$$Thus  the equality $\dim B = \dim T_{\mathcal F} B$ cannot  occur for $h^1 (N) \geq 1$; therefore, $N$ must be non-special.  In this case, $\dim B = \dim T_{\mathcal F} B$ holds if and only if $b=0$ and $N$ is general of its degree.  Consequently, the Proposition is proved.
\end{proof}

Thus, the only remaining case is the following:

\begin{proposition}\label{lem:i=1.1} Let $B$ be any reduced component of $B^{k_2}_d $,  
with $\dim B \geq \rho_d^{k_2}$.  Assume that a general element
$\mathcal F$ of $B$ fits in the following exact sequence;
\begin{equation}\label{exact_seq_b}
 0\to N\to \mathcal F\to L \to 0,  
\end{equation}where $h^1(\mathcal F)=2$ and $h^1(L)=1$. Then, $B$ coincides with the component $B_{\rm reg}$ as in Sect.\;\ref{ss:regular}.
\end{proposition}
\begin{proof}  We will use  notation as 
in \eqref{degree1}.  Since  $B$ is irreducible, all integers in \eqref{degree1}  are constant for a general $\mathcal F \in B$. Then $\frac{3g-1}{2} \leq \delta\le 2g-2$,  since $L$ is special and $\mathcal F$ is semistable.  Hence
\begin{equation}\label{degn}
g-1\le \deg N=d-\delta\le d/2 \le 2g-3\nu .
\end{equation}
By  \eqref{exact_seq_b}, the line bundle $N$ is special and the corresponding coboundary map $\partial$ is of corank one.  As in the proof of Proposition \ref{lem:i=2.1}, for a suitable open dense subset $S$ of $W^{r-1}_{2g-2 +\delta -d}\times C^{(2g-2-\delta)}$,
 one has a projective bundle $\mathbb P(\mathcal E)\to S,$  whose general fiber is $\widehat{\mathcal W}_1:= \mathbb P(\mathcal W_1)$, where $ \mathcal W_1:=\{u\in\ext^1(L,N)\ |\ {\rm corank} (\partial_u)\ge1\}$. Then  the component $B$ is the image of $ \mathcal P$  via a dominant rational map $\mathcal P  \stackrel{\pi}{\dasharrow}  B \subset B_{d}^{k_2}$  (cf.\;\cite[Sect.\;6]{CF} for details).  Hence 
$$ \dim  B \leq \dim W^{r-1}_{2g-2-d+\delta}+2g-2-\delta+\dim \widehat{\mathcal W}_1.$$
 Since from \eqref{degn} $\deg (K_C-N)\le g-1$,  by Martens'  theorem \cite[Thm.\;(5.1)]{ACGH} we obtain
$$\dim W^{r-1}_{2g-2+\delta-d}\le
\begin{cases}
 2g-2-d+\delta=\deg(K_C-N)  &\text{ if } r=1 \\
 2g-2-d+\delta-2r+1 \ &\text{ if }  r\geq 2.
\end{cases} $$
Note that  $m\geq g+2\delta-d -1$  by \eqref{eq:m}, where $m:=\dim ({\rm Ext}^1(L,N))$. Thus it follows  that  $l\ge r$ and $m\ge l+1$ since    $l=h^0(L)=\delta-g+2\ge \frac{g+3}{2}$ and $r-1 \leq \frac{\deg (K_C-N)}{2}$. Applying  Theorem \ref{CF5.8}, we get
$ \dim \widehat{\mathcal W}_1 = m-l+r-2 =  m - \delta +g +r - 4$, whence
\begin{eqnarray*} \dim  B &\le& \dim W^{r-1}_{2g-2-d+\delta} +(2g-2-\delta ) +  m - \delta +g +r - 4\\
&\le&\begin{cases}
5g -d-\delta-7+ m  &\text{ if }  r=1 \\
5g -d-\delta -r-7+ m\ &\text{ if }  r \geq 2,
\end{cases}
\end{eqnarray*}

Assume that  $r\ge 2$; this implies that $N$ cannot be isomorphic to $L$. Therefore  \eqref{eq:m} gives $m=2\delta -d+g-1$.  Thus we have
$$ \rho^{k_2}_d \leq \dim  B\le 6g-2d+\delta-r-8, $$ 
which cannot occur since  $\rho^{k_2}_d=8g-2d-11$ and $\delta \leq 2g-2$.
Therefore, we must have $r= 1$. Then by  \eqref{eq:m}  we get 
\begin{equation}\label{dimB}
\dim B \leq
\begin{cases}
\ (5g -d-\delta-7)+2\delta-d+g-1\ &\text{ if } L\ncong N \\
\ (5g -d-\delta-7)+g\ &\text{ if } L\cong  N.
\end{cases}
\end{equation}
If $L\cong N$ then we have $8g-2d-11 \leq \dim B \leq 6g-d-\delta -7$ which yields  $\deg N =d-\delta \geq 2g -4$. This is a contradiction to \eqref{degn}. Accordingly, we have $L\ncong N$ and hence  by \eqref{dimB}
$$8g-2d-11 \leq \dim B \leq 6g-2d+\delta -8,$$
which implies $\delta \geq 2g-3$. 
Since $L$ is a special line bundle,  it turns out that either $L\simeq K_C$ or $L\simeq K_C(-p)$ for some $p\in C$.

If $L\simeq K_C$,  let $\Gamma$ be the section of the ruled surface $F = \mathbb{P}(\mathcal F)$ corresponding to the quotient $\mathcal  F \to \!\!\!\! \to K_C$; then 
$\dim  |\mathcal {O}_{F}(\Gamma)| = 1$ by \cite[(2.6)]{CF} and the fact that $h^1(\mathcal F) =2$. By  \cite[Prop.\;2.12]{CF} any such $\mathcal F$ admits therefore $K_C-p$ as a quotient line bundle, for some $p \in C$. This completes the proof since $N$ is special. 
\end{proof}

\begin{remark}\label{rem:min} {\normalfont (i) From the proof of Proposition \ref{lem:i=1.1}, 
it also follows that $K_C-p$ is minimal among special quotient line bundles for $\mathcal F$ general in the component $B_{\rm reg}$, completely proving Theorem \ref{thm:main2} (i).

\medskip

\noindent
(ii) Notice moreover that, from the same proof, $\mathcal F$ general in  $B_{\rm reg}$ admits also a {\em presentation} via a canonical quotient, i.e. $0 \to K_C-D-p \to \mathcal F \to K_C\to 0 $, which on the other hand is not via a quotient line bundle of $\mathcal F$ of minimal degree among special quotients and whose residual presentation coincides with that in the proof of \cite[Theorem]{Teixidor1}, 
i.e. $0 \to \mathcal O_C \to \mathcal  E \to  L \to 0$, where $\mathcal E = \omega_C \otimes \mathcal F^*$ and $L = \mathcal O_C(D+p)$. In other words, the component  $B_{\rm reg}$ coincides with that in  \cite[Theorem]{Teixidor1}; the minimality of $K_C-p$ for $\mathcal F$ reflects in our presentation Theorem \ref{thm:main} (i)  via a special section of $\mathcal E$  whose zero locus is of degree one.

}
\end{remark}

We now consider the case $h^1(\mathcal F) = i \ge 3$.

\begin{proposition} There is no reduced component of $B_d^{k_2}$ whose general member $\mathcal F$ is of speciality $i \ge 3$. 
\end{proposition}
\begin{proof} If $\mathcal F \in B_d^{k_2}$ is such that $h^1(\mathcal F) = i \ge 3$, 
then by the Riemann-Roch theorem $h^0(\mathcal F) = d - 2g + 2 + i = k_2 +(i-2) = k_i > k_2$. Thus $\mathcal F \in {\rm Sing} (  B_d^{k_2})$ (cf. \cite[p.\;189]{ACGH}). Therefore the statement follows. 
\end{proof}


%
%

\end{document}